\documentclass[final,leqno,letterpaper]{etna}

\setbibdata{1}{xx}{46}{2022} 

\hypersetup{%
    pdftitle={Conditioning of linear systems arising from penalty methods},
    pdfauthor={William Layton, Shuxian Xu},
    pdfkeywords={penalty method, effective condition number}
    }
\usepackage{amsmath}
\usepackage{graphicx}
\usepackage{amssymb}
\title{Conditioning of linear systems arising from penalty methods\thanks{%
Partially supported by NSF Grant DMS 2110379.
}}

\author{William Layton\footnotemark[2]
        \and Shuxian Xu\footnotemark[3]}

\shorttitle{CONDITIONING OF LINEAR SYSTEMS ARISING FROM PENALTY METHODS} 
\shortauthor{W.~LAYTON AND S.~XU}

\begin{document}

\maketitle

\renewcommand{\thefootnote}{\fnsymbol{footnote}}

\footnotetext[2]{ Department of Mathematics, University of Pittsburgh, Pittsburgh, PA 15260,
USA({\tt wjl@pitt.edu})}
\footnotetext[3]{Department of Mathematics, University of Pittsburgh, Pittsburgh, PA 15260,
USA({\tt shx34@pitt.edu})}

\begin{abstract}
Penalizing incompressibility in the Stokes problem leads, under mild
assumptions, to matrices with condition numbers $\kappa =\mathcal{O}%
(\varepsilon ^{-1}h^{-2})$, $\varepsilon =$\ penalty parameter $<<1$, and $%
h= $ \ meshwidth $<1$. Although $\kappa =\mathcal{O}(\varepsilon ^{-1}h^{-2})
$ is large, practical tests seldom report difficulty in solving these
systems. In the SPD case, using the conjugate gradient\ method, this is
usually explained by spectral gaps occurring in the penalized coefficient
matrix. Herein we point out a second contributing factor. Since the solution
is approximately incompressible, solution components in the eigenspaces
associated with the penalty terms can be small. As a result, the \textit{%
effective} condition number can be much smaller than the standard condition
number.
\end{abstract}



\section{Estimating conditioning}

\begin{center}
\textit{We dedicate this paper to Professor Owe Axelsson.}
\end{center}

Penalty methods have advantages and disadvantages. Disadvantages include the
need to pick a specific value of the penalty parameter $\varepsilon $ and
ill-conditioning of the associated linear system. We show herein that,
measured by the system's effective condition number, this ill-conditioning
is not severe, Theorem 2.1 below.\ This observation is developed herein for
the standard penalty approximation to the Stokes problem%
\begin{equation*}
-\triangle u+\nabla p=f(x)\ \ \text{and} \ \ \nabla \cdot u=0,
\end{equation*}
in a polygonal domain $\Omega $ with boundary condition $u=0$ on $\partial
\Omega $ and normalization $\int_{\Omega }pdx=0$. Let $(\cdot ,\cdot
),||\cdot ||$ denote the $L^{2}(\Omega )$ inner product and norm and $|\cdot
|$\ the usual euclidean norm of a vector and matrix. $X^{h}$ denotes a
standard, conforming, velocity finite element space of continuous, piecewise
polynomials that vanish on $\partial \Omega $. The penalty approximation
results from replacing $\nabla \cdot u=0$\ by $\nabla \cdot u^{\varepsilon
}=-\varepsilon p^{\varepsilon }$ and eliminating $p^{\varepsilon }$. It is:
find $u^{\varepsilon }\in X^{h}$ satisfying%
\begin{equation}
(\nabla u^{\varepsilon },\nabla v)+\varepsilon ^{-1}(\nabla \cdot
u^{\varepsilon },\nabla \cdot v)=(f,v)\text{ for all }v\in X^{h}.
\label{eq:PenaltyStokeseqns}
\end{equation}%
Picking a basis $\left\{ \phi _{1},\cdot \cdot \cdot ,\phi _{N}\right\} $
for $X^{h}$ leads to a linear system with coefficient matrix%
\begin{equation*}
A_{ij}=(\nabla \phi _{i},\nabla \phi _{j})+\varepsilon ^{-1}(\nabla \cdot
\phi _{i},\nabla \cdot \phi _{j}),i,j=1,\cdot \cdot \cdot ,N.
\end{equation*}%
Standard condition number estimates for this system yield bounds like $%
\kappa \leq C\varepsilon ^{-1}h^{-2}$.

Recall that the standard condition number, \cite{A96,W63},
introduced in \cite{NG47,T48} and still of interest, e.g., \cite%
{TV10}, measuring the correlation between relative error and relative
residual, the distance to the nearest singular matrix and the difficulty,
cost and worse-case accuracy in solving a linear system with $A$, is $\kappa
:=|A||A^{-1}|$. However, estimates of the above with $\kappa $\ are often
(but not always, \cite{B98}) too rough because they do not consider the size
of solution components in the matrix eigenspaces. Thus, generalizations
exist, such as the Kaporin condition number \cite{K94} and extensions based
on generalized inverses \cite{D86}, that can be used to obtain better
predictions. One important, easy to compute and interpret generalization is
the \textit{condition number at the solution}, also called the \textit{%
effective condition number}.

\begin{definition}
Let $Ac=b$. Then, $\kappa (c)$, the condition number at the solution $c$, is%
\begin{equation*}
\kappa (c):=|A^{-1}|\frac{|Ac|}{|c|}.
\end{equation*}
\end{definition}

Clearly $\kappa (c)\leq \kappa $ and $\kappa (c)$ takes into account both
the spectrum and the magnitude of solution components across eigenspaces. It
is also known, e.g. \cite{A96}, that the relative error in approximations is
bounded by $\kappa (c)$\ times the relative residual. To our knowledge, this
extension is due to Chan and Foulser \cite{C88}. It was soon thereafter
developed by Axelsson \cite{A95,A96} and has been further developed
in important work in \cite{AK00,AK01,B99,LH08,LHH08,CH94}.

Section 2 gives the proof that $\kappa (u^{\varepsilon })<<\kappa $. Section
3 presents consistent numerical tests.

\section{Analysis of $\protect\kappa (u^{\protect\varepsilon })$}

We assume $X^{h}$\ satisfies the following two assumptions typical, e.g. 
\cite{BS08,C02}, of finite element spaces on quasi-uniform meshes.

A1: [Inverse estimate] For all $v\in X^{h},$ $||\nabla v||\leq Ch^{-1}||v||.$

A2: [Norm equivalence] Let $v=\sum_{i=1}^{N}a_{i}\phi
_{i}(x),N=Ch^{-d},d=\dim (\Omega )=2$ or $3$. Then $||v||$ and $\sqrt{\frac{1%
}{N}\sum_{i=1}^{N}a_{i}^{2}}$ are uniform-in-$h$ equivalent norms.

\begin{theorem}
Let A1 and A2 hold. Select $|\cdot |$ to be the euclidean norm. Let $f^{h}$
be the projection of $f$ into the finite element space and $u^{\varepsilon }$%
\ the solution of (\ref{eq:PenaltyStokeseqns}). Then,%
\begin{eqnarray*}
\max_{f^{h}}\kappa (u^{\varepsilon }) &=&\kappa \leq C(h^{-2}+\varepsilon
^{-1}h^{-2}) \\
\kappa (u^{\varepsilon }) &\leq &C\frac{||f^{h}||}{||u^{\varepsilon }||}%
\text{, and} \\
\kappa (u^{\varepsilon }) &\leq &Ch^{-2}\left( 1+\frac{h}{\varepsilon }\frac{%
||\nabla \cdot u^{\varepsilon }||}{||u^{\varepsilon }||}\right) .
\end{eqnarray*}
\end{theorem}

\begin{proof}
We first estimate $|A^{-1}|^{2}=\max_{b}|A^{-1}b|^{2}/|b|^{2}.$ Given an
arbitrary right-hand side $\overrightarrow{b}$, let the linear system for
the undetermined coefficients $(c_{1},c_{2},\cdot \cdot \cdot ,c_{N})^{T}=%
\overrightarrow{c}$ be denoted $A\overrightarrow{c}=\overrightarrow{b}.$ We
convert in a standard way this linear system to an equivalent formulation
similar to (\ref{eq:PenaltyStokeseqns}). Recall that the mass matrix
associated with this basis is 
\begin{equation*}
M_{ij}=(\phi _{i},\phi _{j}),i,j=1,\cdot \cdot \cdot ,N.
\end{equation*}%
The matrix $M$ is symmetric and positive definite. Under A1 and A2, its
eigenvalues satisfy $0<C_{1}N\leq \lambda (M)\leq C_{2}N$. Given the vector $%
\overrightarrow{b}$ let $\overrightarrow{a}=M^{-1}\overrightarrow{b}$. By
construction 
\begin{equation*}
g(x)=\sum_{i=1}^{N}a_{i}\phi _{i}(x)\text{ satisfies }(g(x),\phi
_{j})=b_{j},j=1,\cdot \cdot \cdot ,N.
\end{equation*}%
By A2, $||g||$ and $(N^{-1}\sum a_{i}^{2})^{1/2}$\ are uniformly equivalent
norms. The bounds on $\lambda (M)$ (also from A2) imply $(N^{-1}\sum
a_{i}^{2})^{1/2}$\ and $(N^{-1}\sum b_{i}^{2})^{1/2}$\ are also uniformly
equivalent norms. Next define 
\begin{equation*}
w=\sum_{i=1}^{N}c_{i}\phi _{i}(x).
\end{equation*}%
By A2 again $||w||$, $(N^{-1}\sum c_{i}^{2})^{1/2}$are equivalent norms.
Thus,%
\begin{equation*}
|A^{-1}|^{2}=\max_{b}\frac{|A^{-1}b|^{2}}{|b|^{2}}\leq C\max_{g\in X^{h}}%
\frac{||w||^{2}}{||g||^{2}}.
\end{equation*}%
By construction $w,g\in X^{h}$\ satisfy%
\begin{equation*}
(\nabla w,\nabla v)+\varepsilon ^{-1}(\nabla \cdot w,\nabla \cdot v)=(g,v)%
\text{ for all }v\in X^{h}.
\end{equation*}%
Setting $v=w$ and using simple inequalities gives $||w||^{2}\leq C||g||^{2}$%
. Indeed,%
\begin{equation*}
C||w||^{2}\leq ||\nabla w||^{2}+\varepsilon ^{-1}||\nabla \cdot w||^{2}=(g,w)%
\text{ }\leq \frac{C}{2}||w||^{2}+C||g||^{2}.
\end{equation*}%
Thus $||w||^{2}\leq C||\nabla w||^{2}\leq C||g||^{2}$ and $%
||w||^{2}/||g||^{2}\leq C$. This implies $|A^{-1}|\leq C$, uniformly of $h,$ 
$\varepsilon $.

Next we estimate $|A\overrightarrow{c}|/|\overrightarrow{c}|$ where $%
u^{\varepsilon }=\sum_{i=1}^{N}c_{i}\phi _{i}(x)$ is the solution of (\ref%
{eq:PenaltyStokeseqns}). By norm equivalence, as above, and (\ref%
{eq:PenaltyStokeseqns}) this is equivalent to $||f^{h}||/||u^{\varepsilon
}|| $ where $f^{h}$\ is the $L^{2}$ projection of $f(x)$ into the finite
element space. This gives the first estimate $\kappa (u^{\varepsilon })\leq
C||f^{h}||/||u^{\varepsilon }||$. For the second estimate, we have%
\begin{eqnarray*}
||f^{h}|| &=&\max_{v\in X^{h}}\frac{(f^{h},v)}{||v||}=\max_{v\in X^{h}}\frac{%
(\nabla u^{\varepsilon },\nabla v)+\varepsilon ^{-1}(\nabla \cdot
u^{\varepsilon },\nabla \cdot v)}{||v||} \\
&\leq &\max_{v\in X^{h}}\frac{||\nabla u^{\varepsilon }||||\nabla
v||+\varepsilon ^{-1}||\nabla \cdot u^{\varepsilon }||||\nabla \cdot v||}{%
||v||} \\
\text{(using A1)} &\leq &\max_{v\in X^{h}}\frac{Ch^{-2}||u^{\varepsilon
}||||v||+\varepsilon ^{-1}Ch^{-1}||\nabla \cdot u^{\varepsilon }||||v||}{%
||v||} \\
&\leq &Ch^{-2}||u^{\varepsilon }||+\varepsilon ^{-1}Ch^{-1}||\nabla \cdot
u^{\varepsilon }||\text{, which implies} \\
\frac{||f^{h}||}{||u^{\varepsilon }||} &\leq &Ch^{-2}+C\varepsilon
^{-1}h^{-1}\frac{||\nabla \cdot u^{\varepsilon }||}{||u^{\varepsilon }||}.
\end{eqnarray*}%
This implies%
\begin{equation*}
\kappa (u^{\varepsilon })\leq C\frac{||f^{h}||}{||u^{\varepsilon }||}\leq
C\left( h^{-2}+\varepsilon ^{-1}h^{-1}\frac{||\nabla \cdot u^{\varepsilon }||%
}{||u^{\varepsilon }||}\right) .
\end{equation*}%
Using $||\nabla \cdot u^{\varepsilon }||\leq C||\nabla u||\leq
Ch^{-1}|||u^{\varepsilon }||$ and A1 yields the standard estimate $\kappa
\leq C(h^{-2}+\varepsilon ^{-1}h^{-2}).$
\end{proof}

\section{An Illustration}

The result in Section 2 gives 3 estimates of conditioning. We explore if the
3 estimates%
\begin{gather*}
\kappa \leq C(h^{-2}+\varepsilon ^{-1}h^{-2}),\text{ \ \ }\kappa
(u^{\varepsilon })\leq C\frac{||f^{h}||}{||u^{\varepsilon }||},\text{ and} \\
\kappa (u^{\varepsilon })\leq C\left( h^{-2}+\varepsilon ^{-1}h^{-1}\frac{%
||\nabla \cdot u^{\varepsilon }||}{||u^{\varepsilon }||}\right) 
\end{gather*}%
yield significantly different predictions. We investigate this question for
2 test problems and 2 finite element spaces. The first test problem has a
smooth solution and the second has an $f(x,y)$ oscillating as fast as
possible on the given mesh. The domain is the unit square and the mesh is a
uniform mesh of squares each divided into 2 right triangles. The first
finite element space is P1, conforming linear elements. The second is P2,
conforming quadratics. The space of conforming linears does not contain a
divergence-free subspace. Thus the coefficient matrix $A$ is expected to
show greater ill-conditioning as $\varepsilon \rightarrow 0$\ than with
conforming quadratics. The above constants "$C$" are $\mathcal{O}(1)$
constants, independent of $\varepsilon $ and $h$. Thus, we compute and
compare the RHS' below%
\begin{eqnarray*}
Est.1 &\simeq &h^{-2}+\varepsilon ^{-1}h^{-2}, \\
Est.2 &\simeq &\text{ }\frac{||f^{h}||}{||u^{\varepsilon }||}\text{ and } \\
Est.3 &\simeq &\varepsilon ^{-1}h^{-1}\frac{||\nabla \cdot u^{\varepsilon }||%
}{||u^{\varepsilon }||}.
\end{eqnarray*}%
We computed these values starting with $\varepsilon =1,h=1$ and halving each
until $\varepsilon =5.96E-08$ and $h=0.00195312.$ We present the results
below selected from equi-spaced $\varepsilon $-values in this data.

\textbf{Test problem 1:} We solve this test problem with P1, conforming
linear elements, and then with P2, conforming, quadratic elements. Choose
$f(x,y)=\left( \sin \left( x+y\right) ,\cos \left( x+y\right) \right) ^{T}$.
For P1 elements we have the following results for $Est.2$.
\begin{gather*}
\begin{array}{c|cccccc}
{\small \varepsilon \Downarrow \setminus h\Rightarrow }\text{ \ } & {\small %
0.125} & {\small 6.2E-2} & {\small 3.1E-2} & \text{ }{\small 1.6E-2} & 
{\small 7.8E-3} & {\small 3.9E-3} \\ 
\hline
{\small 1} & {\small 39} & {\small 38} & {\small 37} & {\small 37} & {\small %
37} & {\small 37} \\ 
{\small 6.3E-2} & {\small 1.9E2} & {\small 1.7E2} & {\small 1.7E2} & {\small %
1.7E2} & {\small 1.7E2} & {\small 1.7E2} \\ 
{\small 3.9E-3} & {\small 1.8E3} & {\small 7.8E2} & {\small 4.3E2} & {\small %
3.3E2} & {\small 3.0E2} & {\small 2.9E2} \\ 
{\small 2.4E-4} & {\small 2.5E4} & {\small 7.5E3} & {\small 2.2E3} & {\small %
8.0E2} & {\small 4.3E2} & {\small 3.3E2} \\ 
{\small 1.5E-5} & {\small 4.0E5} & {\small 1.1E5} & {\small 2.9E4} & {\small %
7.5E3} & {\small 2.1E3} & {\small 8.0E2} \\ 
{\small 9.5E-7} & {\small 6.4E6} & {\small 1.8E6} & {\small 4.5E5} & {\small %
1.1E5} & {\small 2.9E4} & {\small 7.4E3} \\ 
{\small 6.0E-8} & {\small 1.0E8} & {\small 2.9E7} & {\small 7.2E6} & {\small %
1.8E6} & {\small 4.5E5} & {\small 1.1E5}\\
\end{array}
\\
\text{ Table 1: Values of }Est.2\simeq \text{ }\frac{||f^{h}||}{%
||u^{\varepsilon }||}\text{, P1 elements}
\end{gather*}%
Down the columns (fixing $h$ and decreasing $\varepsilon $), the data for\ $%
Est.2$ shows that this quantity grows as $\varepsilon \downarrow 0$, roughly
like $\varepsilon ^{-1}$\ for fixed $h$. Across the rows, for fixed $%
\varepsilon $\ and $h\downarrow 0$, $Est.2$ decreases. Diagonals (necessary
if we choose $\varepsilon \sim h$) show mild growth. Comparing the last row
of the table with the row of $Est.1$ values shows that $Est.2$\ consistently
yields a lower estimate of ill-conditioning than $Est.1$. Concerning the
growth of $Est.2$ as $\varepsilon \downarrow 0$ for fixed $h$, since $h,||f||
$ do not change as $\varepsilon $ varies, this growth is due to $%
||u^{\varepsilon }||\rightarrow 0$ as $\varepsilon \downarrow 0$. The P1
finite element space does not have a non-trivial divergence free subspace, 
\cite{JLMNR17}. Thus, $\varepsilon \downarrow 0$\ forces $||\nabla \cdot%
u||\rightarrow 0$\ which forces $||u||\rightarrow 0$. This indicates that
ill-conditioning of penalty methods \textit{with P1 elements} as $%
\varepsilon \downarrow 0$ is an essential feature caused by the \textit{lack
of a divergence-free subspace}. For comparison with the last row in Table 1,
the values of\ $Est.1\simeq h^{-2}+\varepsilon ^{-1}h^{-2}$ for $\varepsilon
=6.0E-8$ are 
\begin{equation*}
\begin{array}{cccccc}
\text{ }{\small 1.0E9} & \text{ }{\small 4.3E9}\text{ } & {\small 1.7E10} & 
{\small 6.9E10} & {\small 2.7E11} & {\small 1.1E12}%
\end{array}%
.
\end{equation*}%
Clearly $Est.2$\ provides a smaller estimate than $Est.1$. Another
interpretation is that a poor choice of finite element space (made to
accentuate ill-conditioning) makes the problem of selecting $\varepsilon $
acute.

Table 2 presents the analogous results for $Est.3$ and P1 elements.%
\begin{gather*}
\begin{array}{c|cccccc}
{\small \varepsilon \Downarrow \setminus h\Rightarrow } & {\small 0.125} & 
{\small 6.2E-2} & {\small 3.1E-2} & \text{ }{\small 1.6E-2} & {\small 7.8E-3}
& {\small 3.9E-3} \\ 
\hline
{\small 1} & {\small 25} & {\small 50} & {\small 1.0E2} & {\small 2.0E2} & 
{\small 4.0E2} & {\small 8.0E2} \\ 
{\small 6.3E-2} & {\small 3.5E2} & {\small 6.6E2} & {\small 1.3E3} & {\small %
2.5E3} & {\small 5.0E3} & {\small 1.0E4} \\ 
{\small 3.9E-3} & {\small 4.2E3} & {\small 4.1E3} & {\small 4.5E3} & {\small %
6.3E3} & {\small 1.1E4} & {\small 2.0E4} \\ 
{\small 2.4E-4} & {\small 6.2E4} & {\small 4.9E4} & {\small 4.2E4} & {\small %
4.5E4} & {\small 3.5E2} & {\small 5.1E4} \\ 
{\small 1.5E-5} & {\small 9.8E5} & {\small 7.6E5} & {\small 6.3E5} & {\small %
4.2E4} & {\small 6.0E5} & {\small 6.4E5} \\ 
{\small 9.5E-7} & {\small 1.5E7} & {\small 1.2E7} & {\small 1.0E7} & {\small %
9.5E6} & {\small 9.4E6} & {\small 9.4E6} \\ 
{\small 6.0E-8} & {\small 2.5E8} & {\small 1.9E8} & {\small 1.6E8} & {\small %
1.5E8} & {\small 1.5E8} & {\small 1.5E8}%
\end{array}
\\
\text{Table 2: Values of }Est.3\simeq \varepsilon ^{-1}h^{-1}\frac{||\nabla
\cdot u^{\varepsilon }||}{||u^{\varepsilon }||}\text{ , P1 elements}
\end{gather*}%
For $Est.3$\ the data shows similar behavior to $Est.2$ except for fixed $%
\varepsilon $ and $h\downarrow 0$. In this case, ill-conditioning for small $%
\varepsilon $\ and $h\downarrow 0$\ is over-estimated compared with Table 1.
This is expected because the data comes from P1 elements and $||\nabla \cdot
u^{\varepsilon }||$\ occurs in $Est.3$.\ Again, $Est.3$\ still provides a
smaller estimate of ill-conditioning than $Est.1$.

We have attributed the ill-conditioning observed above as $\varepsilon
\rightarrow 0$ to the use of P1 elements. To test if this explanation is
plausible we repeated this test with P2 elements. Since this finite element
space contains a non-trivial divergence free subspace, \cite{JLMNR17}, our
intuition is that ill-conditioning would be reduced. Table 3 below presents $%
Est.3\simeq \varepsilon ^{-1}h^{-1}\frac{||\nabla \cdot u^{\varepsilon }||}{%
||u^{\varepsilon }||}$ values. $Est.2\simeq ||f^{h}||/||u^{\varepsilon }||$
values, not presented, were significantly smaller than $Est.3$ and converged
to $1.13$\ as $\varepsilon ,h\rightarrow 0$.%

\begin{gather*}
\begin{array}{c|cccccc}
{\small \varepsilon \Downarrow \setminus h\Rightarrow } & {\small 0.125} & 
{\small 6.2E-2} & {\small 3.1E-2} & \text{ }{\small 1.6E-2} & {\small 7.8E-3}
& {\small 3.9E-3} \\
\hline
{\small 1} & {\small 25} & {\small 50} & {\small 1.0E2} & {\small 2.0E2} & 
{\small 4.0E2} & {\small 8.0E2} \\ 
{\small 6.3E-2} & {\small 3.1E2} & {\small 6.2E2} & {\small 1.2E3} & {\small %
2.5E3} & {\small 5.0E3} & {\small 1.0E4} \\ 
{\small 3.9E-3} & {\small 6.7E2} & {\small 1.3E3} & {\small 2.5E3} & {\small %
5.0E3} & {\small 1.0E4} & {\small 2.0E4} \\ 
{\small 2.4E-4} & {\small 7.1E2} & {\small 1.3E3} & {\small 2.6E3} & {\small %
5.1E3} & {\small 1.0E4} & {\small 2.0E4} \\ 
{\small 1.5E-5} & {\small 7.1E2} & {\small 1.3E3} & {\small 2.6E3} & {\small %
5.2E3} & {\small 1.0E4} & {\small 2.0E4} \\ 
{\small 9.5E-7} & {\small 7.1E2} & {\small 1.3E3} & {\small 2.7E3} & {\small %
5.3E3} & {\small 1.1E4} & {\small 2.1E4} \\ 
{\small 6.0E-8} & {\small 7.1E2} & {\small 1.3E3} & {\small 2.7E3} & {\small %
5.3E3} & {\small 1.1E4} & {\small 2.1E4}%
\end{array}
\\
\text{Table 3: Values of }Est.3\simeq \varepsilon ^{-1}h^{-1}\frac{||\nabla
\cdot u^{\varepsilon }||}{||u^{\varepsilon }||}\text{, P2 elements}
\end{gather*}%
The above data indicates that with P2 elements and a problem with a smooth
solution, the contribution of the penalty term to ill-conditioning is small.

\textbf{Test problem 2:} We take 
\begin{equation*}
f(x,y)=\left( 10+50\sin \left( \frac{2\pi x}{h}+\frac{2\pi y}{h}\right)
,10+50\sin \left( \frac{2\pi x}{h}+\frac{2\pi y}{h}\right) \right) ^{T}.
\end{equation*}%
Each component oscillates as rapidly as the given mesh allows, see Figure %
\ref{FigF(x,y)} below for $h=0.125$.
\begin{figure}[htp]
\begin{center}
\includegraphics[height=2.5381in,
width=3.0933in]{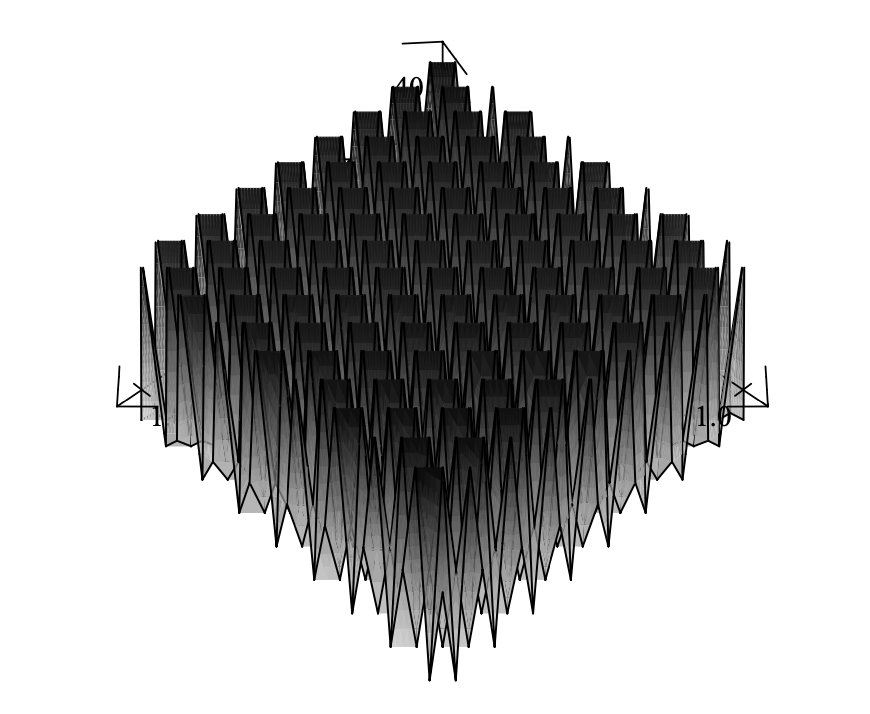}
\end{center}
\caption{Plot of $10+50\sin(\frac{2\pi x}{0.125}+\frac{2\pi y}{0.125})$.}%
\label{FigF(x,y)}%
\end{figure}
The forcing function and thus the
solution change with the mesh size. 
\begin{gather*}
\begin{array}{c|cccccc}
{\small \varepsilon \Downarrow \setminus h\Rightarrow } & {\small 0.125} & 
{\small 6.2E-2} & {\small 3.1E-2} & \text{ }{\small 1.6E-2} & {\small 7.8E-3}
& {\small 3.9E-3} \\ 
\hline
{\small 1} & {\small 1.1E2} & {\small 1.0E2} & {\small 1.0E2} & {\small 1.0E2%
} & {\small 1.0E2} & {\small 1.0E2} \\ 
{\small 6.3E-2} & {\small 5.6E2} & {\small 5.6E2} & {\small 5.6E2} & {\small %
5.6E2} & {\small 5.6E2} & {\small 5.6E2} \\ 
{\small 3.9E-3} & {\small 7.4E3} & {\small 7.5E3} & {\small 7.6E3} & {\small %
7.6E3} & {\small 7.7E3} & {\small 7.7E3} \\ 
{\small 2.4E-4} & {\small 1.1E5} & {\small 1.1E5} & {\small 1.2E5} & {\small %
1.2E5} & {\small 1.2E5} & {\small 1.2E5} \\ 
{\small 1.5E-5} & {\small 1.8E6} & {\small 1.8E6} & {\small 1.8E6} & {\small %
1.8E6} & {\small 1.8E6} & {\small 1.9E6} \\ 
{\small 9.5E-7} & {\small 2.9E7} & {\small 2.9E7} & {\small 2.9E7} & {\small %
2.9E7} & {\small 2.9E7} & {\small 2.9E7} \\ 
{\small 6.0E-8} & {\small 4.7E8} & {\small 4.6E8} & {\small 4.6E8} & {\small %
4.6E8} & {\small 4.6E8} & {\small 4.6E8}%
\end{array}\\
\text{Table 3: Values of }Est.2\simeq \frac{||f^{h}||}{||u^{\varepsilon }||}
\text{, P1 elements}
\end{gather*}
For comparison with the last row in Table 3, the values of\ $Est.1\simeq
h^{-2}+\varepsilon ^{-1}h^{-2}$ for $\varepsilon =6.0E-8$ are%
\begin{equation*}
\begin{array}{cccccc}
{\small 1.1E9} & {\small 4.3E9} & {\small 1.7E10} & {\small 6.9E10} & 
{\small 2.7E11} & {\small 1.1E12}%
\end{array}.
\end{equation*}
In Table 3, as $\varepsilon \downarrow 0,$ $Est.2$\ grows roughly like $%
\varepsilon ^{-1}$. The stable behavior of $Est.2$ as $h\downarrow 0$\ is
unexpected. The behavior of $Est.3\simeq \varepsilon ^{-1}h^{-1}\frac{%
||\nabla \cdot u^{\varepsilon }||}{||u^{\varepsilon }||}$\ for P1 elements
was similar, so not presented.

We now present tests of P2 elements for this problem. In this test $Est.3$
was larger than $Est.2,$ so we present only $Est.3$ data.%
\begin{gather*}
\begin{array}{c|cccccc}
{\small \varepsilon \Downarrow \setminus h\Rightarrow } & {\small 0.125} & 
{\small 6.2E-2} & {\small 3.1E-2} & \text{ }{\small 1.6E-2} & {\small 7.8E-3}
& {\small 3.9E-3} \\ 
\hline
{\small 1} & {\small 25} & {\small 49} & {\small 1.0E2} & {\small 2.0E2} & 
{\small 4.0E2} & {\small 7.9E2} \\ 
{\small 6.3E-2} & {\small 3.8E2} & {\small 7.5E2} & {\small 1.5E3} & {\small %
3.0E3} & {\small 6.0E3} & {\small 1.2E4} \\ 
{\small 3.9E-3} & {\small 6.0E3} & {\small 1.2E4} & {\small 2.4E4} & {\small %
4.8E4} & {\small 9.6E4} & {\small 1.9E5} \\ 
{\small 2.4E-4} & {\small 9.6E4} & {\small 1.9E5} & {\small 3.8E5} & {\small %
7.7E5} & {\small 1.5E6} & {\small 3.1E6} \\ 
{\small 1.5E-5} & {\small 1.5E6} & {\small 3.1E6} & {\small 6.2E6} & {\small %
1.2E7} & {\small 2.5E7} & {\small 4.9E7} \\ 
{\small 9.5E-7} & {\small 2.5E7} & {\small 4.9E7} & {\small 9.8E7} & {\small %
2.0E8} & {\small 3.9E8} & {\small 7.9E8} \\ 
{\small 6.0E-8} & {\small 3.9E8} & {\small 7.9E8} & {\small 1.6E9} & {\small %
3.2E9} & {\small 6.3E9} & {\small 1.3E10}%
\end{array}
\\
\text{Table 4: Values of }Est.3\simeq \varepsilon ^{-1}h^{-1}\frac{||\nabla
\cdot u^{\varepsilon }||}{||u^{\varepsilon }||}\text{ },\text{ P2 elements}
\end{gather*}%
For comparison with the last row in Table 4,\ $Est.1$ values for $%
\varepsilon =6.0E-8$ are%
\begin{equation*}
\begin{array}{cccccc}
{\small 1.1E9} & {\small 4.3E9} & {\small 1.7E10} & {\small 6.9E10} & 
{\small 2.7E11} & {\small 1.1E12}%
\end{array}%
.
\end{equation*}%
For this problem, down the columns (fixed $h$, $\varepsilon \downarrow 0$ )
the computed estimate of $\kappa (u^{\varepsilon })$\ grows roughly like $%
\varepsilon ^{-1}$. Reading across the columns, the values in each row grow
like $h^{-1}$\ (not $h^{-2}$). In all cases, $Est.3$ was smaller than $Est.1$%
.

\section{Conclusions}

The classical view is that two significant issues are impediments to penalty
methods giving low cost and highly accurate velocity approximates. The first
is ill-conditioning of the resulting system matrix. We have shown that the
effective condition number $\kappa (u^{\varepsilon })$ is much smaller than
the usual condition number due to the magnitude of the components in the
penalized eigenspaces being small in a precise sense. Motivated by this
theoretical result, we then compared the derived estimates of
ill-conditioning on two test problems and for two elements. With P1
elements, ill-conditioning was not as severe as $\mathcal{O}(\varepsilon
^{-1}h^{-2})$\ but followed $\varepsilon ^{-1}$ growth as $\varepsilon
\rightarrow 0$. For P2 elements and approximating a smooth solution, $\kappa
(u^{\varepsilon })$\ was smaller the expected $\kappa \sim h^{-2}$ for the
discrete Laplacian. With both P1 and P2 elements, and an academic test
problem with data oscillating as fast as the mesh allows, ill-conditioning
was not as severe as the expected $\mathcal{O}(\varepsilon ^{-1}h^{-2})$\
but also followed the $\varepsilon ^{-1}$ pattern as $\varepsilon
\rightarrow 0$. 

The second significant issue is the difficulty in the selection of an
effective value of $\varepsilon $. While the most commonly recommended, e.g. 
\cite{CK82,HV95}, choices are $\varepsilon =$ \textit{time step,
mesh width}, and (\textit{machine epsilon)}$^{1/2}$, none have proven
reliably effective. Recent work of \cite{KXX21, X21} may resolve this
impediment by an algorithmic, self-adaptive selection of $\varepsilon $
based on some indicator of violation of incompressibility. The estimates in
Theorem 2.1 give insight into the resulting conditioning when this is done.
Let $TOL$ denote a specified tolerance. If $\varepsilon $ is adapted so that
the penalized solution $u^{\varepsilon }$\ satisfies $\frac{||\nabla \cdot
u^{\varepsilon }||}{||u^{\varepsilon }||}\leq TOL,$ then, Theorem 2.1
immediately implies $\kappa (u^{\varepsilon })$ satisfies%
\begin{equation}
\kappa (u^{\varepsilon })\leq Ch^{-2}\left( 1+h\frac{TOL}{\varepsilon }%
\right) .  \label{EstAdaptive1}
\end{equation}%
If $\varepsilon $ is adapted so that the penalized solution $u^{\varepsilon }
$\ satisfies $\frac{||\nabla \cdot u^{\varepsilon }||}{||\nabla
u^{\varepsilon }||}\leq TOL,$ then, similarly, Theorem 2.1 implies $\kappa
(u^{\varepsilon })$ satisfies%
\begin{equation}
\kappa (u^{\varepsilon })\leq Ch^{-2}\left( 1+\frac{TOL}{\varepsilon }%
\right) .  \label{EstAdaptive2}
\end{equation}%
For (\ref{EstAdaptive2}), rewrite $||\nabla \cdot u^{\varepsilon
}||/||u^{\varepsilon }||$ as $\left( ||\nabla \cdot u^{\varepsilon
}||/||\nabla u^{\varepsilon }||\right) \left( ||\nabla u^{\varepsilon
}||/||u^{\varepsilon }||\right) $. The inverse estimate, A1, implies $%
||\nabla u^{\varepsilon }||/||u^{\varepsilon }||\leq Ch^{-1}$, yielding (\ref%
{EstAdaptive2}).

The numerical tests suggest that two factors not considered in Theorem 2.1
are significant. The first is whether the finite element space has a
nontrivial, divergence-free subspace. The second is the influence of
smoothness of the sought solution or its problem data on effective
conditioning. In addition, highly refined meshes are used in practical flow
simulations and the linear system often has large skew symmetric part. The
extension of the analysis herein to include these effects is an open problem.

\bibliographystyle{siam}

\end{document}